%% file: main.tex
\documentclass[twoside,11pt,a4paper]{article}

\input{packages.tex}

\usepackage{tikz-cd}

\DeclareNameAlias{default}{last-first}

\input{theorem.tex}

\input{commands.tex}
\loadglsentries{glossary.tex}


\begin{document}

\begin{center}
{\fontsize{14}{20}\bf Derivations on Group Algebras with Coding Theory Applications}\end{center}

\begin{center}
Leo Creedon, Kieran Hughes\\
\bigskip
Institute of Technology Sligo, Ireland\\
\bigskip
\texttt{creedon.leo@itsligo.ie} \\
\texttt{kieran.hughes@mail.itsligo.ie} 
\end{center}

\begin{abstract}
    This paper classifies the derivations of group algebras in terms of the generators and defining relations of the group.
    If $RG$ is a group ring, where $R$ is commutative and $S$ is a set of generators of $G$ then necessary and sufficient conditions on a map from $S$ to $RG$ are established, such that the map can be extended to an $R$-derivation of $RG$.
    Derivations are shown to be trivial for semisimple group algebras of abelian groups. 
    The derivations of finite group algebras are constructed and listed in the commutative case and in the case of dihedral groups.
    In the dihedral case, the inner derivations are also classified.
    Lastly, these results are applied to construct well known binary codes as images of derivations of group algebras.
    
\end{abstract}

\section{Introduction}\label{sectintro}
\input{KIntro.tex}

\input{Section1.tex}

\section{Derivations of Group Rings}\label{sectFGGR}
\input{Section2.tex}

\section{Applications}\label{sectApplications}
We will now apply the results of the previous sections to finite commutative group algebras in Section~\ref{subsectCommutative} and then to finite dihedral group algebras in Section~\ref{subsectDihedral}.
The study of finite group algebras is motivated in part by applications to coding theory which appear in Section~\ref{subsectCoding}, where the extended binary Golay $[24,12,8]$ code and the extended binary quadratic residue $[48,24,12]$ code are presented as images of derivations of group algebras.

\subsection{Derivations of Commutative Group Algebras}\label{subsectCommutative}
\input{Section3_1.tex}

\subsection{Derivations of Dihedral Group Algebras}\label{subsectDihedral}
\input{Section3_2.tex}

\subsection{Applications to Coding Theory}\label{subsectCoding}
\input{Section3_3.tex}


\vspace{0.3cm}

\clearpage

{\bf Acknowledgements.} This research has received funding from the Institute of Technology Sligo President’s Bursary Award and the Institute of Technology Sligo Capacity Building Fund. The authors also thank Martin Mathieu for useful conversations on the subject of derivations and the reviewers for their helpful comments.

\printbibliography

\end{document}

%% file: packages.tex
\usepackage[utf8]{inputenc}
\usepackage{hyperref}
\usepackage{amssymb}
\usepackage{amsmath}
\usepackage{verbatim}
\usepackage{mathabx}

\usepackage[backend=biber,style=numeric]{biblatex}
\renewbibmacro{in:}{}
\renewbibmacro*{volume+number+eid}{%
  \printfield{volume}%
  \setunit*{\addnbspace}
  \printfield{number}%
  \setunit{\addcomma\space}%
  \printfield{eid}}
\DeclareFieldFormat[article]{number}{\mkbibparens{#1}}

\usepackage{amsthm}
\usepackage{enumerate}
\usepackage{enumitem}
\usepackage{mathtools}
\usepackage{braket}
\usepackage{color}
\usepackage{MnSymbol}
\usepackage{float}

\usepackage [english]{babel}
\usepackage [autostyle, english = american]{csquotes}
\MakeOuterQuote{"}

\usepackage[nopostdot,style=super,nonumberlist, toc]{glossaries}
\DeclarePairedDelimiter\floor{\lfloor}{\rfloor}

%% file: theorem.tex
\makeatletter
\newcommand{\newreptheorem}[2]
{\newtheorem*{rep@#1}{\rep@title}\newenvironment{rep#1}[1]
{\def\rep@title{#2 \ref*{##1}}
\begin{rep@#1}}
{\end{rep@#1}}}
\makeatother

\theoremstyle{plain}
\newreptheorem{theorem}{Theorem}
\newreptheorem{lemma}{Lemma}
\newtheorem*{theorem*}{Theorem}
\newtheorem{theorem}{Theorem}[section]

\newtheorem{lemma}[theorem]{Lemma}
\newtheorem{corollary}[theorem]{Corollary}
\theoremstyle{definition}
\newtheorem{definition}[theorem]{Definition}
\newtheorem{example}[theorem]{Example}
\newtheorem{remark}[theorem]{Remark}
\newtheoremstyle{case}{}{}{}{}{}{:}{ }{}
\theoremstyle{case}

\newcounter{casecount}
\usepackage{etoolbox}
\AtBeginEnvironment{proof}{\setcounter{casecount}{0}}


%% file: commands.tex
\newcommand*{\f}{\mathbb{F}}

\newcommand*{\fpg}{\mathbb{F}_{p^n}G}

\providecommand{\abs}[1]{\lvert#1\rvert}
\renewcommand{\restriction}{\mathord{\upharpoonright}}


%% file: glossary.tex
\newglossaryentry{N}
{
    name=\ensuremath{\mathbb{N}},
    sort=N,
    description={The set of natural numbers, $\{1, 2, 3, \dots \}$}
}
 
\newglossaryentry{Z}
{
    name=\ensuremath{\mathbb{Z}},
    sort=Z,
    description={The set of integers, $\{\dots, -2, -1, 0, 1, 2, \dots \}$}
}

\newglossaryentry{fpn}
{
    name=\ensuremath{\mathbb{F}_{p^n}},
    sort=Fpn,
    description={A finite field of order $p^n$, for a prime $p$ and $n \in \gls{N}$}
}

\newglossaryentry{lcm}
{
    name=\ensuremath{lcm(m,n)},
    sort=lcm,
    description={The lowest common multiple of the integers $m$ and $n$}
}
\newglossaryentry{derR}
{
    name=\ensuremath{Der(R)},
    sort=derR,
    description={The set of all derivations on a ring $R$}
}
\newglossaryentry{ZR}
{
    name=\ensuremath{Z(R)},
    sort=ZR,
    description={The center of the ring $R$}
}
\newglossaryentry{d}
{
    name=\ensuremath{d},
    sort=der,
    description={A derivation on a ring}
}
\newglossaryentry{zdotd}
{
    name=\ensuremath{z \cdot d},
    sort=zdotd,
    description={The derivation defined by $a \mapsto zd(a)$, for all $d \in Der(R), \  a \in R$ and $z \in Z(R)$}
}
\newglossaryentry{Kerd}
{
    name=\ensuremath{Ker_d},
    sort=Kerd,
    description={The kernel of the additive group homomorphism $d, \ Ker_d = \{a \in R \mid d(a)=0\}$}
}
\newglossaryentry{C}
{
    name=\ensuremath{C},
    sort=C,
    description={The ring of constants, $C = \{ c \in R \mid d(c) = 0 \text{ for all } d \in Der(R)\}$ }
}
\newglossaryentry{1}
{
    name=\ensuremath{1},
    sort=1,
    description={The multiplicative identity element in a ring}
}
\newglossaryentry{0}
{
    name=\ensuremath{0},
    sort=0,
    description={The additive identity element in a ring}
}
\newglossaryentry{K}
{
    name=\ensuremath{K},
    sort=K,
    description={A finite field of unspecified order}
}
\newglossaryentry{iso}
{
    name=\ensuremath{\simeq},
    sort=iso,
    description={Is isomorphic to}
}
\newglossaryentry{px}
{
    name=\ensuremath{\partial_x},
    sort=partialx,
    description={The partial derivation with respect to $x$}
}
\newglossaryentry{RG}
{
    name=\ensuremath{RG},
    sort=RG,
    description={The group ring with group $G$ and ring $R$}
}
\newglossaryentry{absG}
{
    name=\ensuremath{\abs{G}},
    sort=abs,
    description={The order of the group $G$}
}
\newglossaryentry{C_n}
{
    name=\ensuremath{C_n},
    sort=Cn,
    description={The cyclic group of order $n$}
}
\newglossaryentry{O_n}
{
    name=\ensuremath{O_n},
    sort=On,
    description={The ring formed from $\{m a \mid m = 0, 1, \dots , n-1\}$ such that the additive group is $C_n$ and all product are zero, that is $a^2=0$}
}

\newglossaryentry{deg(a)}
{
    name=\ensuremath{deg(a)},
    sort=deg,
    description={The degree of the polynomial $a$}
}
\newglossaryentry{injection}
{
    name=\ensuremath{\hookrightarrow},
    sort=inj,
    description={An injection}
}
\newglossaryentry{d^n}
{
    name=\ensuremath{d^n},
    sort=dn,
    description={The map that is the composition of $d$ with itself $n$ times}
}
\newglossaryentry{Ad}
{
    name=\ensuremath{\mathcal{A}_d},
    sort=Ad,
    description={The attractor of the derivation $d$}
}
\newglossaryentry{per}
{
    name=\ensuremath{\mathcal{P}},
    sort=P,
    description={The map that returns the period of an element of an attractor}
}
\newglossaryentry{TrestrictW}
{
    name=\ensuremath{T\restriction_{W}},
    sort=restr,
    description={The restriction of the linear transformation $T$ to the $T$-invariant subspace $W$}
}
\newglossaryentry{D}
{
    name=\ensuremath{D},
    sort=D,
    description={The matrix representing the derivation $d$ as a linear transformation}
}
\newglossaryentry{nkd}
{
    name=\ensuremath{[n, k, \delta]},
    sort=nkd,
    description={A code of length $n$, dimension $k$ and minimum distance $\delta$}
}
\newglossaryentry{DB}
{
    name=\ensuremath{[D]_B},
    sort=DB,
    description={The matrix $D$ with respect to the basis $B$}
}
\newglossaryentry{Xi}
{
    name=\ensuremath{X_i},
    sort=Xi,
    description={A cyclic group with presentation $X_i = \langle x_i \mid x_i^{p^{m_i}} = 1 \rangle$, for $m_i \in \gls{N}$ and a prime $p$}
}
\newglossaryentry{Xbarl}
{
    name=\ensuremath{\overline{X}_l},
    sort=Xl,
    description={$X_1 \times X_2 \times \dots \times X_{l-1} \times X_{l+1} \times \dots \times X_t$}
}
\newglossaryentry{B}
{
    name=\ensuremath{\mathcal{B}},
    sort=B,
    description={A set of elements that form a basis for the vector space}
}
\newglossaryentry{Rinfty}
{
    name=\ensuremath{R_{\infty}},
    sort=Ri,
    description={The generalised range space of a linear transformation}
}
\newglossaryentry{Ninfty}
{
    name=\ensuremath{N_{\infty}},
    sort=Ni,
    description={The generalised null space of a linear transformation}
}
\newglossaryentry{CRA}
{
    name=\ensuremath{C_{R}(A)},
    sort=CRA,
    description={The centraliser of $A$ in the ring $R$}
}

%% file: KIntro.tex
Group rings and derivations of rings have both been studied for more than 60 years. For a history of group rings see  Polcino Milies and Sehgal \cite{Polcino} and for a survey article on derivations see Ashraf, Ali, and Haetinger \cite{Haetinger}. 
The results of Posner \cite{posner1957derivations} and Herstein \cite{herstein1978note} attracted particular attention.  
Prime, semiprime and 2-torsion free rings were a focus of the resulting research.

Derivations of $C^*$-algebras have been studied by several authors. In~\cite{sakai1971derivations}, Sakai proved that every derivation of a simple $C^*$-algebra becomes inner in its multiplier algebra.
Mathieu and Villena, in~\cite{mathieu2003structure} study the structure of Lie derivations on $C^*$-algebras.
In the 2000 paper Derivations on Group Algebras \cite{ghahramani2000derivations},
Ghahramani, Runde and Willis, examine the first cohomology space of the group algebra $L^1(G)$, where $G$ is a locally compact group.
The derivation problem asks whether every derivation from $L^1(G)$ to $M(G)$ is inner, where $G$ is a locally compact group and $M(G)$ is the multiplier algebra of $L^1(G)$. It was solved by Losert \cite{Losert}. 
The 2017 preprint "Derivations of Group Algebras", \cite{arutyunov2017derivations} by 
Arutyunov, Mishchenko and Shtern describes the outer derivations of $L^1(G)$.

Group rings have been used to construct new codes as well as to  study existing codes. 
In \cite{hurley2009codes} Hurley and Hurley present techniques for constructing codes from group rings. 
The codes constructed consist primarily of two types, zero-divor codes and unit-derived codes. 
The structure of group ring codes is examined in \cite{hughes2001structure}. 
The author gives a decomposition of a group ring code into twisted group ring codes and proves the nonexistence of self-dual group ring codes in particular cases.

Derivations have also been employed in coding theory. 
In \cite{boucher2014linear} codes are constructed as modules over skew polynomial rings, where the multiplication is defined by a derivation and an automorphism. 
In this paper we are primarily concerned with derivations of group algebras and their application to coding theory.  

However, there has not been as much research into derivations of group algebras with positive characteristic. Notable exceptions include Smith \cite{smith1978derivations}, Spiegel \cite{spiegel1994derivations}  and Ferrero, Giambruno and Polcino Milies \cite{PolcinoNote}. In the latter paper the authors prove the following theorem.
\begin{theorem}\label{PolcinoNote}\cite{PolcinoNote}
Let $R$ be a semiprime ring and $G$ a torsion group such that $[G : Z(G)] < \infty$, where $Z(G)$ denotes the center of $G$. Suppose that either char $R = 0$ or for every characteristic $p$ of $R,\ p \notdivides o(g)$, for all $g \in G$. Then every $R$-derivation of $RG$ is inner.
\end{theorem}
In this paper we are particularly interested in finite group algebras. This is motivated in part by applications to error correcting codes.
Theorems~\ref{PolcinoNote} and \ref{deronpregular} direct our focus, in the commutative case, to the study of derivations of modular (nonsemisimple) group algebras with positive characteristic.

Theorem~\ref{dKZero} shows that when $K$ is an algebraic extension of a prime field all derivations of a $K$-algebra are $K$-derivations.
If $RG$ is a group ring, where $R$ is commutative and $S$ is a set of generators of $G$ then necessary and sufficient conditions on a map $f \colon S \to RG$ are established, in Theorem~\ref{CompatabilityThm}, such that $f$ can be extended to an $R$-derivation of $RG$.
Section~\ref{sectApplications} outlines some applications of  
the results of Section~\ref{sectFGGR}. 
All derivations of finite commutative group algebras of positive characteristic are determined in Theorem~\ref{ThmDerFiniteAbelG}. 
If $G$ is a finite abelian group and $K$ a finite field of positive characteristic $p$ then the image of a minimum set of generators of the Sylow $p$-subgroup of $G$ under a derivation of $KG$ can be chosen arbitrarily, however this is not always the case in the noncommutative setting.
An inner derivation of a ring $R$ maps $a \in R$ to $ab - ba$, for some element $b \in R$.
In the case of finite dihedral group algebras of characteristic 2, a basis is given for the space of derivations in Theorem~\ref{Thm:BasisDer(fpmD2n)} and also for those that are inner in Theorem~\ref{Thm:basisofInn}.
  
The extended binary Golay $[24,12,8]$ code and the extended binary quadratic residue $[48,24,12]$ code are both presented as images of derivations of group algebras in Section~\ref{subsectCoding}.

%% file: Section1.tex
\begin{definition}
Notation: $\mathbb{N}, \ \mathbb{Z}$ and $\mathbb{Q}$ denote the natural numbers, the integers and the rational numbers, and $\f_{p^n}$ denotes the finite field with $p^n$ elements. The group ring $RG$ denotes the set of all formal linear combinations of the form $\sum_{g \in G} a_g g$, of finite support where $a_g \in R$, together with the operations of addition (componentwise) and multiplication defined as $(\sum_{g \in G} a_g g) (\sum_{h \in G} b_h h) = \sum_{g,h \in G} a_g b_h gh$. We adopt the usual convention that empty sums are $0$ and empty products are $1$.
\end{definition}

\begin{definition}\label{derdef}
A \emph{derivation} of a ring $R$ is a mapping $d \colon R \to R$ satisfying 
\begin{gather}
d(a+b) = d(a) + d(b), \quad \quad \text{ for all } a, b \in R.
\label{eq1}\\
d(ab) = d(a)b + ad(b), \quad \quad \text{ for all } a, b \in R.
\label{eq2}
\end{gather}
Equation~(\ref{eq2}) is known as Leibniz's rule. 
Write $\gls{derR}$ for the set of derivations of a ring $R$. Note that if $R$ is a unital ring then $d(1) = 0$, since $d(1) = d\big(1(1)\big) = d(1)1 + 1d(1)$.
\end{definition}
\begin{definition}\label{defsdotd}
Let $d \in \gls{derR}$ and $r\in R$ for a ring $R$. Then the map 
$r \cdot d \colon R \to R$ is defined as $a \mapsto rd(a)$ for all $a \in R$. 
\end{definition}

\begin{lemma}\label{ZRModule}
Let $Z$ be a central subring of a ring $R$. Then \gls{derR} together with the action $\cdot$ is a $Z$-module.
\end{lemma}

\begin{definition}
Let $RG$ be a group ring. Then a derivation $d \colon RG \to RG$ is an \emph{R-derivation} if $d(R) = \{0\}$.
\end{definition}

\begin{definition}
Given a ring $R$ and $a, b \in R$, define the \emph{Lie commutator} $[a, b] = ab - ba$.
A derivation $d$ on a ring $R$ is \emph{inner} if for all $a \in R$ we have $d(a) = ab - ba$ for some $b \in R$. 
In this case we write $d=d_b$.
\end{definition}

%% file: Section2.tex
In this section we establish necessary and sufficient conditions on a map $f \colon S \to RG$, such that $f$ can be extended to an $R$-derivation of the group ring $RG$, where $S$ is a set of generators of $G$ and $R$ is commutative. 
First, some identities and preliminary results are presented. 

\begin{lemma}\label{LogTablesTheorem}
Let $d$ be a derivation of a ring $R$. Then
\begin{flalign}
&(i)\ \  \quad d(\textstyle \prod_{i=1}^{m} a_i) = \sum_{i = 1}^{m} \left( (\prod_{j = 1}^{i-1}a_j) d(a_i) (\prod_{j = i+1}^{m}a_j) \right), \quad \text{ for all } a_i \text{ in }R.& \label{EqDiffRuleNonComm2} \\
&(ii)\ \quad d(a^m) = \textstyle \sum_{i=0}^{m-1}a^{i}d(a)a^{(m - 1 - i)}, \quad \text{for all } a \in R \text{ and } m \in \gls{N}.& \label{EqDiffRuleNonComm} \\
&(iii) \quad \textstyle  \sum_{i=0}^{n-1}a^{i}d(a)a^{(n - 1 - i)} = 0, \quad \text{ for all units } a  \text{ in $R$ of order } n.& \label{EqDiffRuleNonCommUnit} \\
&(iv)\  \quad d(a^k) = ka^{k-1}d(a), \quad \text{for all } a \in R \text{ which commute with } d(a) \text{ and } k \in \gls{N}.& \label{EqDiffRuleComm} \\
&(v)\ \  \quad d(a^k) = ka^{k-1}d(a), \quad \text{ for all } 
\text{ units } a \in R \text{ which commute with }d(a) \text{ and } k \in \gls{Z}. & \label{EqDiffRuleCommUnit}
\end{flalign}
\end{lemma}
\begin{proof}
$(i)$ We will prove Equation~\ref{EqDiffRuleNonComm2} by induction on $m$. 
Base case: $m = 1$. This is true as $d(a_1) = \sum_{i = 1}^{1} 1 d(a_1) 1$. 
Assume that $d(\prod_{i=1}^{m} a_i) = \sum_{i = 1}^{m} \left( ( \prod_{j = 1}^{i-1}a_j) d(a_i) ( \prod_{j = i+1}^{m}a_j) \right)$. Then
\[d(\prod_{i=1}^{m+1} a_i) = d(\prod_{i=1}^{m} a_i)a_{m+1} + \Big( \prod_{i=1}^{m} a_i\Big) d(a_{m+1}) = \sum_{i = 1}^{m} \Big( ( \prod_{j = 1}^{i-1}a_j) d(a_i) ( \prod_{j = i+1}^{m}a_j) \Big) a_{m+1} + \Big(\prod_{i=1}^{m} a_i\Big) d(a_{m+1})\]
\[= \sum_{i = 1}^{m+1} \Big( ( \prod_{j = 1}^{i-1}a_j) d(a_i) ( \prod_{j = i+1}^{m+1}a_j) \Big).\]
Therefore Equation~\ref{EqDiffRuleNonComm2} holds for all $m \in \gls{N}$. 

$(ii)$ Let $a_i = a$ in Equation~\ref{EqDiffRuleNonComm2}. Then for all $m \in \gls{N}$
\[d(a^{m}) = \sum_{i = 1}^{m} \Big( ( \prod_{j = 1}^{i-1}a) d(a) ( \prod_{j = i+1}^{m}a) \Big) = \sum_{i=1}^{m}a^{i-1}d(a)a^{(m-i)} = \sum_{i=0}^{m-1}a^{i}d(a)a^{(m-1-i)} .\]
$(iii)$ Setting $m = n$ in Equation~\ref{EqDiffRuleNonComm} implies 
\[0 = d(1) = d(a^n) = \sum_{i=0}^{n-1}a^{i}d(a)a^{(n - 1 - i)}.\]
$(iv)$ Let $a$ be an element of $R$ that commutes with $d(a)$. Then using Equation~\ref{EqDiffRuleNonComm} 
\[d(a^k) = \sum_{i=0}^{k-1}a^{i}d(a)a^{(k - 1 - i)} = \sum_{i=0}^{k-1}a^{k-1}d(a) = ka^{k-1}d(a).\]

$(v)$ Let $a$ be a unit which commutes with $d(a)$. Then $a^{-1}$ is also a unit which commutes with $d(a)$ since $a^{-1}d(a) = a^{-1}d(a)aa^{-1} = a^{-1}ad(a)a^{-1} = d(a)a^{-1}$. Therefore $0 = d(1) = d(a^{-1}a) = d(a^{-1})a + a^{-1}d(a)$ and so $d(a^{-1}) = -a^{-1}d(a)a^{-1} = -a^{-2}d(a)$. Moreover, $a^{-1}$ also commutes with $d(a^{-1})$ since $a^{-1}d(a^{-1}) = a^{-1}(-a^{-2}d(a)) = -a^{-2}d(a)a^{-1} = d(a^{-1})a^{-1}$. Therefore for any positive integer $k$ 
\[d(a^{-k}) = d((a^{-1})^k) = k(a^{-1})^{k-1}d(a^{-1}) = k(a^{-k+1})(-a^{-2}d(a)) = -k(a^{-k-1})d(a).\]
Furthermore, $0 = d(1) = d(a^0) = 0a^{-1}d(a)$ and so Equation~(\ref{EqDiffRuleCommUnit}) holds for all integers $k$.
\end{proof}

The following Theorem shows that when $K$ is an algebraic extension of a prime field all derivations of a $K$-algebra are $K$-derivations.
\begin{theorem}\label{dKZero}
Let $A$ be a $K$-algebra where $K$ is an algebraic extension of a prime field $F$ and let $d \in Der(A)$. 
Then $d(K) = \{ 0\} $ and $d$ is a $K$-linear map.
\end{theorem}
\begin{proof}
Let $d \in Der(A)$. If $char(F) > 0$ then for $b\in F$, $d(b) = d(1 + 1 + \dots + 1) = d(1) + d(1) + \dots + d(1) = bd(1) = b0 = 0$, and so $d(F) = 0$. Let $F = \mathbb{Q}$ and let $a,b \in \mathbb{Z}$ with $b > 0$. 
Note that $0=d(0)=d(1-1)=d(1)+d(-1) = 0+d(-1)$, so $d(-1)=0$.
Then 
$bd(a/b) =  d(a/b) + \cdots +d(a/b) = d(a/b+ \cdots +a/b) = d(a) = \pm d(1+\cdots +1) = \pm (d(1)+\cdots +d(1)) = 0$. Therefore $d(a/b)=0$, so $d(F) = 0$ for all prime fields $F$.

Let $a$ be a nonzero element of $K$ and let $m_a(x) = \sum_{j=0}^{n_a}b_{aj}x^j \in F[x]$ be the minimal polynomial of $a$ over $F$. 
$a$ is a central unit in $K$ and so Equation~\ref{EqDiffRuleCommUnit} of  Lemma~\ref{LogTablesTheorem} applies. 
Note that for $b \in F$ and $\alpha \in K$ we have $d(b \alpha) = bd(\alpha)$, since $d(F) = 0$. 
Thus applying a derivation $d$ to $m_a(a) = 0$ and using Equation~\ref{EqDiffRuleCommUnit} 
\begin{gather*}
0 = d(0) = d(m_a(a)) = d(\sum_{j=0}^{n_a}b_{aj}a^j) = \sum_{j=0}^{n_a}b_{aj}d(a^j)\\ 
= \sum_{j=0}^{n_a}b_{aj} ja^{j-1}d(a) = \Big(\sum_{j=1}^{n_a}b_{aj} ja^{j-1}\Big)d(a) = q(a)d(a),
\end{gather*}
where $q$ is a polynomial in $F[x]$. Moreover, $q(a) \ne 0$ as this would contradict the minimality of the degree of $m_a(x)$. Therefore $d(a) = 0$, since $q(a)$ is invertible as it is a non zero element of the field $K$. Hence $d(K) = \{ 0 \}$.

The $K$-linearity of $d$ is immediate since $d$ is additive and if $a\in A$ and $k\in K$ then $d(ka)= d(k)a + kd(a)= 0+kd(a)$.
\end{proof}

\begin{corollary}\label{Cor:KalgGTorsion}
Let $K$ be an algebraic extension of a prime field $F$. Let $G$ be a torsion group such that
$[G : Z(G)] < \infty$, where $Z(G)$ denotes the center of $G$. 
Suppose that either char$(K) = 0$ or
that char$(K)=p>0$, and $p$ does not divide the order of $g$, for all $g \in G$. Then every derivation of $KG$ is
inner. 
\end{corollary}
\begin{proof}
By Theorem~\ref{dKZero}, every derivation of $KG$ is a $K$-derivation and since every field is semiprime, Theorem~\ref{PolcinoNote} implies that every derivation of $KG$ is inner.
\end{proof}

Note that the requirement that $K$ is algebraic over $F$ is necessary in Theorem~\ref{dKZero} as the following example shows.

\begin{example}
Let $\mathbb{Q}(t)$ be a transcendental extension of the rationals (the field of rational functions of $t$). 
Since $\mathbb{Q}(t)$ is a $\mathbb{Q}$-algebra, Theorem~\ref{dKZero} implies that $d(\mathbb{Q})=\{ 0 \}$ for all derivations $d$ of $\mathbb{Q}(t)$.
However, by Proposition~5.2 of Chapter VIII in \cite{lang2002algebra}, there exists a nonzero derivation $d$ of $\mathbb{Q}(t)$, since $\mathbb{Q}(t)$ is a finitely generated extension over $\mathbb{Q}$ that is not separable algebraic.
\end{example}

\begin{theorem}\label{CompatabilityThm}
Let $G = \langle S \mid T \rangle$ be a group, where $S$ is a generating set and $T$ a set of relators. 
Let $F_S$ be the free group on $S$ and $\phi\colon F_S \to G$
the homomorphism of $F_S$ onto $G$.
Let $R$ be a commutative unital ring and $f$ a map from $S$ to $RG$. \\
Then 
\begin{enumerate}[label=(\roman*)]
    \item $f$ can be uniquely extended to a map $f^*$ from $F_S$ to $RG$ such that 
    \begin{align}\label{GenLeib}
    f^*(uv) = f^*(u)\phi(v) + \phi(u)f^*(v), \quad \text{ for all }u, v \in F_S,
    \end{align}
    \item the map $f $ from $S$ to $RG$ can be extended to an $R$-derivation of $RG$ if and only if $f^*(t) = 0$, for all $t \in T$,
    \item if $f$ can be extended to an $R$-derivation of $RG$, then this extension is unique.
\end{enumerate}
\end{theorem}
\begin{proof}
Let $f$ be a map from $S$ to $RG$.
$\phi$ is the identity map on $S$, so for $s\in S$, $\phi(s^{-1}s)=\phi(s^{-1})\phi(s)=\phi(s^{-1})s=\phi(1)=1$, so $\phi(s^{-1})=s^{-1}$. Thus $\phi$ is the identity map on $S\cup S^{-1}$.
\\
\textit{(i) }
We wish to extend $f$ to $f^* \colon F_S \to RG$, which satisfies Equation~\ref{GenLeib}. 

Define $f^* \colon  F_S \to RG$ as follows:
\begin{equation}\label{DefinitionOff*2}
f^*(w_i) =
\begin{cases}
f(w_i) & \text{ if } w_i \in S, \\
-w_if(w_i^{-1})w_i & \text{ if } w_i \in S^{-1},\\
0 & \text{ if } w_i = 1
\end{cases}
\end{equation}
and letting $ w = \prod_{i = 1}^k w_i$, where $w_i \in S \cup S^{-1}$, define
\begin{equation}\label{DefinitionOff*1}
f^*(w) = \sum_{i = 1}^{k}\Big( (\prod_{j = 1}^{i-1}w_j) f^*(w_i) (\prod_{j = i+1}^{k}w_j)\Big).
\end{equation}

Let $0 \le l \le k$ and $u = \prod_{i = 1}^l w_i$ and $v = \prod_{i = l + 1}^k w_i$. Then by Equations~\ref{DefinitionOff*2} and \ref{DefinitionOff*1}
\[f^*(uv) = \sum_{i = 1}^{k}\Big( ( \prod_{j = 1}^{i-1}w_j) f^*(w_i) (\prod_{j = i+1}^{k}w_j)\Big) \]
\[=\Big(\sum_{i = 1}^{l} (\prod_{j = 1}^{i-1}w_j) f^*(w_i) (\prod_{j = i+1}^{l}w_j)\Big) \prod_{j = l+1}^{k}w_j + 
\prod_{j = 1}^{l}w_j \sum_{i = l+1}^{k} (\prod_{j = l+1}^{i-1}w_j) f^*(w_i) (\prod_{j = i+1}^{k}w_j) \]
\[ =f^*(u) \prod_{j = l+1}^{k}\phi(w_j) + \prod_{j = 1}^{l}\phi(w_j) f^*(v) \]
\[= f^*(u)\phi(v) + \phi(u)f^*(v).\]
Therefore $f^*$ defined by Equations~\ref{DefinitionOff*2} and \ref{DefinitionOff*1} satisfies Equation~\ref{GenLeib}. 

If $w$ is a word on $S$, denote the reduced word by $\overline{w}$.  
In order for $f^*$ to be well defined on $F_S$ we need to show that $f^*(w) = f^*(\overline{w})$ for all words $w$ on $S$.
Let $u, v$ be words on $S$ and let $a \in S$.

Then by Equation~\ref{DefinitionOff*2}, $f^*(a)a^{-1} + af^*(a^{-1}) = f(a)a^{-1} - aa^{-1}f(a)a^{-1} = 0$. 
Similarly, $f^*(a)a^{-1} + af^*(a^{-1}) = 0$ for all $a \in S^{-1}$. Let $a \in S \cup S^{-1}$. Then by Equation \ref{DefinitionOff*1}, $f^*(aa^{-1}) = 0$ and so by Equation~\ref{GenLeib}
\[ f^*(uaa^{-1}v) = 
f^*(u)\phi(aa^{-1}v) + \phi(u)f^*(aa^{-1}v) \] 
\[ = f^*(u)\phi(v) + \phi(u)f^*(aa^{-1})\phi(v) + \phi(uaa^{-1})f^*(v) = f^*(u)\phi(v) + \phi(u)f^*(v) = f^*(uv).\]
Therefore $f^*(w) = f^*(\overline{w})$ for all words $w$ on $S$.
We now prove the uniqueness of $f^*$.

Assume that there exists a map $f_*\colon F_S \to RG$, distinct from $f^*$ which is also an extension of $f$ and which also satisfies Equation~\ref{GenLeib}. 
Let $1$ be the identity element of $F_S$. 
Then $f_*(1) = f_*(1(1)) = f_*(1)1 + 1f_*(1)$, which implies that $f_*(1) = 0 = f^*(1)$.
Let $s \in S$. 
Then $f_*(s) = f(s) = f^*(s)$ and $0 = f_*(s^{-1}s) = f_*(s^{-1})s + s^{-1}f_*(s)$. 
This implies that $f_*(s^{-1}) = -s^{-1}f_*(s)s^{-1} = f^*(s^{-1})$.
Therefore there exists an element $x$ of $F_S$, of positive length $c > 1$, such that $f^*(x) \ne f_*(x)$ and $f^*(z) = f_*(z)$ for all words $z$ in $F_S$ of length less than $c$.
Write $x = \prod_{i=1}^{c} x_i$, where $x_i \in S \cup S^{-1}$.
Thus $f^*(\prod_{i=1}^{c-1} x_i) = f_*(\prod_{i=1}^{c-1} x_i)$ and $f^*(x_c) = f_*(x_c)$, since $\prod_{i=1}^{c-1} x_i$ and $x_c$ are both elements of $F_S$ whose length is less than $c$. 
Therefore by Equation~\ref{GenLeib}
\begin{equation*}
    f_*(x) = f_*(\prod_{i=1}^{c-1} x_i)\phi(x_c) + \phi(\prod_{i=1}^{c-1} x_i)f_*(x_c)= f^*(\prod_{i=1}^{c-1} x_i)\phi(x_c) + \phi(\prod_{i=1}^{c-1} x_i)f^*(x_c) = f^*(x).
\end{equation*}
This contradiction implies that $f^*$ is the unique extension of $f$ to $F_S$, such that $f^*(uv) = f^*(u)\phi(v) + \phi(u)f^*(v)$, for all $u, v \in F_S$.  This proves \emph{(i)}.

\textit{(ii) } Considering $S$ as a subset of $G$, suppose that the map $f \colon S \to RG $ can be extended to an $R$-derivation $d$ of $RG$.
Then for any $s \in S$, $d(s) = f(s)$ and $0 = d(s^{-1}s) = d(s^{-1})s + s^{-1}d(s)$ and so  $d(s^{-1}) = -s^{-1}d(s)s^{-1} = -s^{-1}f(s)s^{-1}$. 
Therefore $d(a)= f^*(a)$, for all $a\in S\cup S^{-1}$ by Equation~\ref{DefinitionOff*2}.
Let $t = \prod_{i=1}^{m}t_i \in T$, where $t_i \in S \cup S^{-1}$ for $i = 1, 2, \dots ,m$. Then by Equations~\ref{DefinitionOff*1} and \ref{EqDiffRuleNonComm2}
\[f^*(t) = \sum_{i = 1}^{m} \Big((\prod_{j = 1}^{i-1}t_j) f^*(t_i) (\prod_{j = i+1}^{m}t_j)\Big) = \sum_{i = 1}^{m} \Big((\prod_{j = 1}^{i-1}t_j) d(t_i) (\prod_{j = i+1}^{m}t_j)\Big)
= d(t) = 0.\]
This proves the implication in \emph{(ii)}.


Conversely, assume $f^*(t) = 0$, for all $t \in T$. 
Let $t \in T$.
Then $\phi(t) = 1$ and $f^*(t^{-1}) = 0$, since $0 = f^*(tt^{-1}) = f^*(t)\phi(t^{-1}) + \phi(t)f^*(t^{-1}) = 0(1) + (1)f^*(t^{-1}) = f^*(t^{-1})$.
Let $\epsilon \in \{1,\ -1\}$.
Then for all $w \in F_S$
\begin{equation}\label{ImageOfConjEqualsZero3}
\begin{gathered}
f^*(w^{-1}t^{\epsilon}w) = f^*(w^{-1})\phi(t^{\epsilon}w) + \phi(w^{-1})f^*(t^{\epsilon}w)  \\
= f^*(w^{-1})\phi(t^{\epsilon}w) + \phi(w^{-1})f^*(t^{\epsilon})\phi(w) + \phi(w^{-1}t^{\epsilon})f^*(w) \\
= f^*(w^{-1})\phi(w) + \phi(w^{-1})f^*(w) = f^*(w^{-1}w) = 0.
\end{gathered}
\end{equation}
Let $N = \langle T ^{F_S} \rangle$ be the normal closure of $T$. Any non-identity element $n$ of $N$ can be written as $\prod_{i = 1}^{k} w_i^{-1} t_i^{\epsilon_i} w_i$, where $w_i \in F_S$,  $t_i \in T$ and $\epsilon_i \in \{-1, \ 1\}$. 
Therefore by Equations~\ref{DefinitionOff*1} and \ref{ImageOfConjEqualsZero3}
\[f^*(n) = f^*\big(\prod_{i = 1}^{k} w_i^{-1} t_i^{\epsilon_i} w_i\big) = 
\sum_{i = 1}^{k} \phi \Big( \prod_{j = 1}^{i-1}w_j^{-1} t_j^{\epsilon_j} w_j\Big) f^*(w_i^{-1} t_i^{\epsilon_i} w_i) \phi \Big( \prod_{j = i+1}^{k}w_j^{-1} t_j^{\epsilon_j} w_j\Big) = 0.\]
Also $\phi(n) = 1$, for all $n \in N$ and so for any $w \in F_S$, $f^*(wn) = f^*(w)\phi(n) + \phi(w)f^*(n) = f^*(w)$.
Let $g, h\in G  = \langle S \mid T \rangle \simeq \displaystyle \frac{F_S}{ \ \langle T^{F_S} \rangle }$ and let $u, v$ be elements of $F_S$, such that $g = \phi(u)$ and $h = \phi(v)$.
Extend $f \colon S \to RG$ to $\hat{f}\colon G \to RG$ by defining $\hat{f}(g)=f^*(u)$.
Then $\hat{f}(gh) = f^*(uv) = f^*(u)\phi(v) + \phi(u)f^*(v) = \hat{f}(g)h + g\hat{f}(h)$.
Suppose $\tilde{f}$ is also an extension of $f$ distinct from $\hat{f}$ that satisfies $\tilde{f}(gh) = \tilde{f}(g)h + g\tilde{f}(h)$ for all $g, h \in G$. 
Let $l \colon G \to \mathbb{N}$ be the minimum length of an element of $G$, defined by $l(g) = min\{k  \mid g=\prod_{i = 1}^k g_i, \ g_i \in S \cup S^{-1}\}$. 
Then there exists an $x \in G$ of minimum length such that $\tilde{f}(x) \ne \hat{f}(x)$.
For all $s \in S$, $0 = \tilde{f}(ss^{-1}) = \tilde{f}(s)s^{-1} + s\tilde{f}(s^{-1})$ and $\tilde{f}(s) = \hat{f}(s)$. 
Thus $\tilde{f}(s^{-1})  = -s^{-1} \hat{f}(s)s^{-1} = -s^{-1} f(s)s^{-1} = f^*(s^{-1}) = \hat{f}(s^{-1})$.
Therefore $\tilde{f}(g) = \hat{f}(g)$ for all $g \in G$ such that $l(g) < 2$ and so $x$ can be written as $x = yz$, where $y, z \in G$ such that $l(y) < l(x)$ and $l(z) < l(x)$. 
Then $\tilde{f}(x) = \tilde{f}(yz) = \tilde{f}(y)z + y\tilde{f}(z) = \hat{f}(y)z + y\hat{f}(z) = \hat{f}(x)$.
This contradiction implies that $\hat{f}$ is the unique extension of $f$ such that $\hat{f}(gh) = \hat{f}(g)h + g\hat{f}(h)$ for any $g, h \in G$. 
Extend $\hat{f}$, $R$-linearly to $RG$ and denote this unique extension also by $\hat{f}$. 
Let $\alpha = \displaystyle\sum_{g \in G} a_g g$ and $\beta =\displaystyle\sum_{h \in G} b_h h$ be elements of $RG$, where $a_g, b_h\in R$.
Then $\hat{f}(\alpha + \beta) = \hat{f}(\alpha) + \hat{f}(\beta)$ as $\hat{f}$ is an $R$-linear map. Moreover
\begin{gather*}
\hat{f}(\alpha) \beta + \alpha \hat{f}(\beta) = \Big( \sum_{g \in G}a_g \hat{f}(g) \Big) \Big( \sum_{h \in G}b_h h \Big) + \Big( \sum_{g \in G} a_g g \Big) \Big( \sum_{h \in G} b_h \hat{f}(h) \Big) \\
= \sum_{g, h} a_g b_h \hat{f}(g) h  + \sum_{g, h} a_g b_h g \hat{f}(h) = 
\sum_{g, h} a_g b_h (\hat{f}(g) h  + g \hat{f}(h))  = \sum_{g, h} a_g b_h \hat{f}(gh) \\
= \hat{f}\Big(\sum_{g, h} a_g b_h gh \Big) = \hat{f}\Big(\sum_{g} a_g g \sum_{h} b_h h\Big) = \hat{f}(\alpha \beta).
\end{gather*}
Therefore the map $\hat{f}$ obeys Leibniz's rule for all products of elements of $RG$ and so is an $R$-derivation of $RG$. This proves \emph{(ii)} and \emph{(iii)}.
\end{proof}

\begin{corollary}\label{CompatabilityCor}
Let $G = \langle S \mid T \rangle$ be a group, where $S$ is a generating set and $T$ a set of relators. 
Let $F_S$ be the free group on $S$ and $\phi\colon F_S \to G$
the homomorphism of $F_S$ onto $G$.
Let $K$ be an algebraic extension of a prime field and $f$ a map from $S$ to $KG$. 
Then 
\begin{enumerate}[label=(\roman*)]
    \item $f$ can be uniquely extended to a map $f^*$ from $F_S$ to $KG$ that satisfies Equation~\ref{GenLeib},
    \item $f$ can be extended to a derivation of $KG$ if and only if $f^*(t) = 0$, for all $t \in T$,
    \item if $f$ can be extended to a derivation of $KG$, then this extension is unique.
\end{enumerate}
\end{corollary}
\begin{proof}
By Theorem~\ref{dKZero} all derivations of $KG$ are $K$-derivations and so the result follows from Theorem~\ref{CompatabilityThm}.
\end{proof}

\begin{remark}
The restriction that $R$ be a commutative ring in Theorem~\ref{CompatabilityThm} is necessary. 
To demonstrate this, let $r_1, r_2$ be noncommuting elements in a ring $R$ and let $G$ be the infinite cyclic group generated by $S = \{s\}$, that is the free group on $S$.
Let $f \colon S \to RG$ be the map defined by $s \mapsto r_1$ and extend $f$ to a map $f^* \colon G \to RG$ as in Theorem~\ref{CompatabilityThm}~\emph{(i)}. 
Assume that $f$ can be extended to an $R$-derivation $d$ of $RG$.
Then
\[d(s)r_2s + sd(r_2s) = r_1r_2s + sr_2d(s) = r_1r_2s + sr_2r_1 = (r_1r_2 + r_2r_1)s.\]
However 
\[d(sr_2s) =  r_2d(s^2) = r_2(r_1s + sr_1) = 2r_2r_1s.\]
Therefore the Leibniz rule does not apply since $d(sr_2s) \ne d(s)r_2s + sd(r_2s)$. This contradicts the assumption that $f$ can be extended to an $R$-derivation of $RG$.
\end{remark}

%% file: Section3_1.tex
The next result directs our study of derivations of commutative group algebras to the nonsemisimple case.
\begin{theorem}\label{deronpregular}
Let $R$ is a commutative unital ring. 
Let $H$ be a torsion central subgroup of a group $G$, where the order of $h$ is invertible in $R$, for all $h \in H$.
Then $d(R)=\{ 0\}$ if and only if $d(RH)=\{ 0\}$, for all $d\in Der(RG)$. 
\end{theorem}
\begin{proof}
Let $d$ be any element of $Der(RG)$. 
Assume that $d(R)=\{ 0\}$.
Let $h$ be an element of $H$ of order $s$. 
Applying $d$ to $h^s=1$ implies $sh^{s-1}d(h)=0$ by Equation~\ref{EqDiffRuleCommUnit} of Lemma~\ref{LogTablesTheorem}.
By assumption $s$ is invertible in $R$ and so $s$ is also invertible in $RG$. Therefore $d(h) = 0$ for any $d \in Der(RG)$. 
Let $\alpha = \displaystyle \sum_{h \in H} a_h h$ be any element of $RH$. 
Then
\[ d(\alpha) = d( \sum_{h \in H} a_h h ) = \sum_{h \in H} d(a_h h) = \sum_{h \in H} a_h d(h) = \sum_{h \in H} a_h(0) = 0,\] 
by Leibniz's rule since $d(R)=\{ 0\} $ and so $d(RH)=\{ 0 \} $. The converse is immediate.
\end{proof}

\begin{corollary}\label{Zd}
    \begin{enumerate}[label=(\roman*)]
        \item Let $G$ be a finite abelian group and $F$ either the rational numbers or an algebraic extension of the rationals. Then $FG$ has no nonzero derivations.
        \item Let $H$ be a $p$-regular subgroup of a finite abelian group $G$ and $F=\f_{p^n}$. Then all derivations of $FG$ are $FH$-derivations.
    \end{enumerate}
\end{corollary}
\begin{proof}
For part \emph{(i)} let $H = G$.
In both cases $F$ is a commutative unital ring and $H$ is a torsion central subgroup of $G$, where the order of $h$ is invertible in $F$ for all $h \in H$. 
Also $d(F)=\{ 0\} $ for all $d\in Der(FG)$, by Theorem~\ref{dKZero}.
Therefore the results follow from Theorem~\ref{deronpregular}.
\end{proof}
Note that \emph{(i)} of this Corollary also follows from Theorem~\ref{PolcinoNote}.

\begin{remark}
In Theorem \ref{deronpregular}, the requirement that the subgroup $H$ is central is necessary. For example, there are 26 non zero derivations of $\f_3 D_8$. Moreover the 27 derivations of $\f_3 D_8$ are inner by Theorem~\ref{PolcinoNote} or Corollary~\ref{Cor:KalgGTorsion}. 
\end{remark}

In Theorem~\ref{ThmDerFiniteAbelG} we determine all derivations of finite commutative group algebras of positive characteristic $p$. 

\begin{theorem}\label{ThmDerFiniteAbelG}
Let $K$ be a finite field of positive characteristic $p$. 
Let $G \simeq H \times X$ be a finite abelian group, where $H$ is a $p$-regular group and $X$ is a $p$-group with the following presentation
\[X = \langle x_1, \dots, x_n \mid x_k^{p^{m_k}} = 1,\ [x_k,x_l]=1,\text{ for all } k, l \in \{1, 2, \dots n\} \rangle, \]
where $n, m_k \in \gls{N}$.
For $ i,j \in \{ 1, \dots, n\} $, let 
$f_i \colon  \{ x_1, \ldots , x_n \} \to KG$ be defined by \\
$f_i (x_j) = \begin{cases} 1 & if \ i = j \ and \\ 0 & otherwise. \end{cases}$ \\
Then $f_i$ can be uniquely extended to a derivation of $KG$ denoted by $\partial_i$.   
Moreover $Der(KG)$ is a vector space over $K$ with basis $\{ g\partial_i \mid g\in G, i=1,\ldots ,n \} $.
\end{theorem}
\begin{proof}
By Corollary~\ref{Zd} \emph{(ii)} all derivations of $KG$ are $KH$-derivations.
Let $S=\{ x_1 , \ldots , x_n \} $ and let $f$ be any map from $S$ to $KG$. 
By Theorem~\ref{CompatabilityThm} $f$ can be uniquely extended to a map $f^* \colon F_S \to KG$ satisfying Equation~\ref{GenLeib}. Moreover, $f$ can be extended to a derivation of $KG$ if and only if $f^*(t) = 0$ for $t \in \{[x_k, x_l], \  x_k^{p^{m_k}} \mid k, l = 1, 2, \dots , n\}$. Let $a, b \in S$. Then 
\[f^*(a^{-1}b^{-1}ab) = f^*(a^{-1})b^{-1}ab + a^{-1}f^*(b^{-1})ab + a^{-1}b^{-1}f^*(a)b + a^{-1}b^{-1}af^*(b)\]
\[= -a^{-1}f(a)a^{-1}b^{-1}ab - a^{-1}b^{-1}f(b)b^{-1}ab + a^{-1}b^{-1}f(a)b + a^{-1}b^{-1}af(b)\]
\[= -a^{-1}f(a) - b^{-1}f(b) + a^{-1}f(a) + b^{-1}f(b) = 0.\]
Therefore $f^*([x_k, x_l]) = 0$, for all $k, l = 1, 2, \dots , n$.
Also by Equation~\ref{DefinitionOff*1} 
\[f^*(x_k^{p^{m_k}}) = \sum_{i = 1}^{p^{m_k}}\Big( (\prod_{j = 1}^{i-1}x_k) f^*(x_k) (\prod_{j = i+1}^{p^{m_k}}x_k)\Big) = p^{m_k} x_k^{(p^{m_k}-1)} f^*(x_k) = 0,\]
since $KG$ has characteristic $p$. 
Therefore any map $f \colon S \to KG$ can be uniquely extended to a derivation of $KG$. 
By Lemma~\ref{ZRModule} $Der(KG)$ is a vector space over $K$. 
Let $B = \{ g\partial_i \mid g\in G, i=1,\ldots ,n \}$. 
Any map $f \colon S \to KG$ can be written as $\sum_{i = 1}^{n} \sum_{g \in G} k_{i, g} g f_i$, where $k_{i, g} \in K$. The extension of $f$ to a derivation of $KG$ is $\sum_{i = 1}^{n} \sum_{g \in G} k_{i, g} g \partial_i$. Therefore any derivation of $KG$ can be written as a $K$-linear combination of the elements of $B$. 
Furthermore, if $(\sum_{i = 1}^{n} \sum_{g \in G} k_{i, g} g \partial_i )(x_j) = 0$, then $\sum_{g \in G} k_{g,j} g   = 0$, which implies $k_{g, j} = 0$ for all $g \in G$.
Therefore the elements of $B$ are $K$-linearly independent and so form a basis of $Der(KG)$.
\end{proof}

\begin{remark}
Derivations of finite commutative group algebras $\fpg$ are either the zero derivation (in the semisimple case by Corollary~\ref{Zd}(ii)) or can be decomposed as in Theorem ~\ref{ThmDerFiniteAbelG} as the sum of derivations of the group algebras of the cyclic direct factors of $G$.

As we will see in the next section, derivations of noncommutative finite group algebras are more involved.
\end{remark}

%% file: Section3_2.tex
Let $n$ be an integer greater than 2 and let $D_{2n}$ denote the dihedral group with $2n$ elements and presentation $\langle x,y \mid x^{n} = y^2 = (xy)^2 = 1 \rangle $.
This section classifies the derivations of the group algebra $\f_{2^m} D_{2n}$.

\begin{definition}
Let $RG$ be a group ring. The \emph{augmentation ideal of RG}, denoted by $\Delta(G)$, is the kernel of the homomorphism from $RG$ to $R$ defined by $\sum_{g \in G} a_g g \mapsto \sum_{g \in G} a_g$. 
\end{definition}

\begin{lemma}\label{LemmaBasisCenterF2D8}\cite[pp.113]{Passman2011}
The centre of the group algebra $KG$ has as a $K$-basis the set of all finite conjugacy class sums. \qed
\end{lemma}

\begin{lemma}\label{basisforZKG}
If $n$ is even, $Z(\f_{2^m} D_{2n})$, the center of $\f_{2^m} D_{2n}$ is a subspace of $\f_{2^m} D_{2n}$ of dimension $\frac{n}{2} + 3$ and a basis
$\{1,\ x^{\frac{n}{2}},\ x^{1} + x^{-1},\ x^{2} + x^{-2},\ \dots,\ x^{\tfrac{n}{2}-1} + x^{-\tfrac{n}{2}+1},\ y + x^{2}y + x^{4}y + \dots + x^{n-2}y,\  xy + x^{3}y + x^{5}y + \dots + x^{n-1}y\}$. 

If $n$ is odd, $Z(\f_{2^m} D_{2n})$ has dimension $\frac{n+3}{2}$ and a basis 
$\{1,\ x^{1} + x^{-1},\ x^{2} + x^{-2},\ \dots,\ x^{\tfrac{n-1}{2}} + x^{\tfrac{-n+1}{2}},\ y + xy + x^2y + \dots + x^{n-1}y\}$.
\end{lemma}
\begin{proof}
If $n$ is even the conjugacy classes of $D_{2n}$ are as follows: $\{1\}$, $\{x^{\frac{n}{2}}\}$, $\{x^{i}, x^{-i}\}$, for $i = 1, 2, \dots, \tfrac{n}{2}-1$, $\{y, x^{2}y, x^{4}y, \dots, x^{n-2}y\}$ and $\{xy, x^{3}y, x^{5}y, \dots, x^{n-1}y\}$. 
If $n$ is odd the conjugacy classes of $D_{2n}$ are as follows: $\{1\}$, $\{x^{i}, x^{-i}\}$, for $i = 1, 2, \dots, \tfrac{n-1}{2}$ and $\{y, xy, x^2y,\dots, x^{n-1}y\}$. 
The result follows from counting the conjugacy classes and by Lemma~\ref{LemmaBasisCenterF2D8}.
\end{proof}

\begin{corollary}\label{BasisOfCy}
Let $C(y)$ and $C(xy)$ denote respectively the centralisers of $y$ and $xy$ in $\f_{2^m} D_{2n}$.
Then the following are bases for $C(y)$ and $C(xy)$. \\
Case (1): n is even
\begin{align*}
B_e(y) = \{1,\ x^{\frac{n}{2}},\ y,\ x^{\frac{n}{2}}y \} \cup \{(x^i + x^{-i}),\ (x^i + x^{-i})y \mid  i = 1, 2, \dots , \tfrac{n}{2}-1\} \\
B_e(xy) = \{1,\ x^{\frac{n}{2}},\ xy,\ x^{\frac{n}{2}}xy \} \cup \{(x^i + x^{-i}),\ x(x^i + x^{-i})y \mid  i = 1, 2, \dots , \tfrac{n}{2}-1\}.
\end{align*}

Case (2): n is odd
\begin{align*}
B_o(y) = \{1,\ y\} \cup \{(x^i + x^{-i}),\ (x^i + x^{-i})y \mid  i = 1, 2, \dots , \tfrac{n-1}{2}\} \\
B_o(xy) = \{1,\ xy\} \cup \{(x^i + x^{-i}),\ x(x^i + x^{-i})y \mid  i = 1, 2, \dots , \tfrac{n-1}{2}\}.
\end{align*}
\end{corollary}
\begin{proof}
Let $g \in D_{2n}$ and denote by $Orb(g^y)$ the subset $\{g, g^y\}$ of $D_{2n}$. The set $\{Orb(g^y) \mid g \in G\}$ is a partition of $D_{2n}$. The set of elements formed by taking the partition sums forms a basis $B_e(y)$ for $C(y)$, when $n$ is even and $B_o(y)$, when $n$ is odd.
The map $\alpha \colon D_{2n} \to D_{2n}$ defined by $y \mapsto xy$ and $x \mapsto x$ is an automorphism of $D_{2n}$. 
Extend $\alpha$ $\f_{2^m}$-linearly to an $\f_{2^m}$-algebra automorphism of $\f_{2^m} D_{2n}$.

Let $c = a + by$, where $a, b \in \f_{2^m} \langle x \rangle$.
Assume that $c \in C(y)$. Then $(a + by)y = y(a + by)$, which implies that $ay = ya$ and $by = yb$ and so $a, b \in Z(\f_{2^m} D_{2n})$. Therefore $\alpha(c) \in C(xy)$, since
\[xy \alpha(c) = xy(a + bxy) = axy + bxyxy = (a + bxy)xy = \alpha(c)xy.\]
Conversely, assume $\alpha(c) = a + bxy \in C(xy)$. Then
\[a^yxy + b^y = xy(a + bxy) = (a + bxy)xy = axy + b.\]
This implies $a = a^y$ and $b = b^y$ and so $c \in C(y)$.
Therefore $c \in C(y)$ if and only if $\alpha(c) \in C(xy)$. Applying $\alpha$ to the basis $B_e(y)$ gives $B_e(xy)$ and applying $\alpha$ to $B_o(y)$ gives $B_o(xy)$.
\end{proof}

\begin{definition}\label{dNotation}
Given a derivation $d$ of $\f_{2^m}D_{2n}$, denote it by $d=d_{x', y'}$, where $x' = d(x)$ and $y' = d(y)$. Note that $d(x)$ and $d(y)$ uniquely determine this derivation.
\end{definition}

By Lemma~\ref{ZRModule}, $Der(\f_{2^m} D_{2n})$ forms a vector space over $\f_{2^m}$. The following Theorem exhibits a basis for $Der(\f_{2^m} D_{2n})$.

\begin{theorem}\label{Thm:BasisDer(fpmD2n)}
If $n$ is even, $Der(\f_{2^m} D_{2n})$ has dimension $2n + 4$ and a basis 
\[ \left\{ d_{x', y'} \mid (x',y') \in \{(\lambda y, 0),\ (x \omega y, \omega) \mid \lambda \in B_e(xy),\ \omega \in B_e(y)\} \right\} . \] 

If $n$ is odd, $Der(\f_{2^m} D_{2n})$ has dimension $\frac{3n+1}{2}$ and a basis 
\begin{flalign*}
& \big\{d_{x', y'} \mid (x',y') \in \\
& \{((x^i + x^{-i})y, 0), ((1+x)y, 1), (0, y), (x(x^i + x^{-i})y, x^i + x^{-i}), (0, (x^i + x^{-i})y) \mid i = 1, \dots , \tfrac{n-1}{2}\} \big\}. 
\end{flalign*} 
\end{theorem}
\begin{proof}
The relators of $D_{2n}$ are $y^2$, $(xy)^2$ and $x^n$. 
Therefore by Corollary~\ref{CompatabilityCor}, $f \colon \{x, y\} \to \f_{2^m} D_{2n}$ can be extended to a derivation of $\f_{2^m} D_{2n}$ if and only if $f^*(y^2) = f^*((xy)^2) = f^*(x^n) = 0$.
$f^*(y^2) = 0$ if and only if $f(y) \in C(y)$. 
Also $f^*((xy)^2) = 0$ if and only if $f(x)y + xf(y) \in C(xy)$, since $f^*((xy)^2) = f^*(xy)xy + xyf^*(xy)$ and $f^*(xy) = f(x)y + xf(y)$. 
We now treat the cases where $n$ is even and $n$ is odd separately.

Case (1): $n$ is even.
$f^*(x^n) = f^*(x^{\frac{n}{2}}x^{\frac{n}{2}}) = f^*(x^{\frac{n}{2}})x^{\frac{n}{2}} + x^{\frac{n}{2}}f^*(x^{\frac{n}{2}}) = 0$, for all $f(x) \in \f_{2^m} D_{2n}$, since $x^{\frac{n}{2}} \in Z(\f_{2^m} D_{2n})$.
Therefore $f \colon \{x, y\} \to \f_{2^m} D_{2n}$ can be extended to a derivation of $\f_{2^m} D_{2n}$ if and only if $f(y) \in C(y)$ and $f(x)y + xf(y) \in C(xy)$. 

Let $f(y)$ and $f^*(xy)$ be arbitrary elements of $C(y)$ and $C(xy)$, respectively. 
Write $f(y) = \Omega = \sum_{i=1}^{n+2} r_i \omega_i$ and $f^*(xy) = \Lambda = \sum_{i=1}^{n+2} k_i \lambda_i$, where $r_i, k_i \in \f_{2^m}$, $\omega_i \in B_e(y)$ and $\lambda_i \in B_e(xy)$.
Then $\Lambda = f^*(xy) = f(x)y + x \Omega$ and so $f(x) = \Lambda y + x \Omega y$. 
Therefore 
\[Der(\f_{2^m} D_{2n}) = \{d_{(\Lambda y + x \Omega y, \Omega)} \} = \{d_{(\sum k_i \lambda_i y + \sum r_i x \omega_i y,\ \sum r_i \omega_i )}\}.\]
Define $B_e = \{d_{(\lambda y, 0)},\ d_{(x \omega y, \omega)} \mid \lambda \in B_e(xy),\ \omega \in B_e(y)\}$. 
Then $B_e$ is a spanning set for $Der(\f_{2^m} D_{2n})$, since $r_1\cdot d_{(x_1,y_1)} + r_2 \cdot  d_{(x_2,y_2)} = d_{(r_1x_1 + r_2x_2, r_1y_1 + r_2y_2)}$ for $r_1, r_2 \in \f_{2^m}$ and $x_1, x_2, y_1, y_2 \in \f_{2^m} D_{2n}$.
We now show that the elements of $B_e$ are linearly independent. Assume
\[ \sum_{i=1}^{n+2} k_i d_{(\lambda_i y, 0)} + \sum_{i=1}^{n+2} r_i d_{(x \omega_i y, \omega_i)} = d_{(\sum k_i \lambda_i y + \sum r_i x \omega_i y,\ \sum r_i \omega_i )} = d_{(0,0)}\]
This implies $r_i = k_i = 0$ for $i = 1, 2, \dots, n + 2$.
Therefore $Der(\f_{2^m} D_{2n})$ has a basis $B_e = \{ d_{x', y'} \mid (x',y') \in \{(\lambda y, 0),\ (x \omega y, \omega) \mid \lambda \in B_e(xy),\ \omega \in B_e(y)\} \} $ and dimension $2n + 4$.

Case (2): $n$ is odd.
Let $f(x) = a + by$, where $a, b \in \f_{2^m} \langle x \rangle$.
Assume that $f$ can be extended to a derivation of $Der(\f_{2^m} D_{2n})$. 
So $f^*(x^n)=0$.
Applying Equation~\ref{DefinitionOff*1} gives
\[0 = \sum_{i = 1}^{n} \Big( (\prod_{j = 1}^{i-1}x) f^*(x)  (\prod_{j = i+1}^{n}x)\Big) = \sum_{t = 0}^{n-1} x^t (a + by) x^{n-1-t} = nax^{n-1} + \sum_{t = 0}^{n-1} x^{2t + 1} by . \]
Right multiplying this equation by $x$ and using $n \equiv 1$~(mod 2) and $\sum_{t = 0}^{n-1} x^{2t} = (\sum_{t = 0}^{n-1} x^{t})^2 = n\sum_{t = 0}^{n-1} x^{t}$ gives $a + \sum_{t = 0}^{n-1} x^{t} by = 0$.
This implies that $a=0$ and $b \in \Delta(\langle x \rangle)$. 
Therefore there is a third condition when $n$ is odd, namely $f(x) = by$, where $b \in \Delta(\langle x \rangle)$.

Let $f(y) = \Omega \in C(y)$ and let $f^*(xy) = \Lambda \in C(xy)$. 
Then $\Lambda = f^*(xy) = f(x)y + x \Omega$ and so $f(x) = \Lambda y + x \Omega y$. 
Therefore $Der(\f_{2^m} D_{2n}) = \{d_{(\Lambda y + x \Omega y, \Omega)} \mid \Lambda \in C(xy),\ \Omega \in C(y), \ \Lambda + x \Omega \in \Delta(\langle x \rangle)\}$. 
Write $\Lambda$ and $\Omega$ as $\f_{2^m}$-linear combinations of $B_o(xy)$ and $B_o(y)$ respectively, that is 
\begin{gather*}
\Lambda = k_1 1 + k_2 xy + \sum_{i=1}^{\frac{n-1}{2}} k_{3,i} (x^i + x^{-i}) + \sum_{i=1}^{\frac{n-1}{2}} k_{4,i} x(x^i + x^{-i})y, \\
\Omega = r_1 1 + r_2 y + \sum_{i=1}^{\frac{n-1}{2}} r_{3,i} (x^i + x^{-i}) + \sum_{i=1}^{\frac{n-1}{2}} r_{4,i} (x^i + x^{-i})y \quad \text{and so} \\
\Lambda + x \Omega = k_1 1 + r_1 x + (k_2 + r_2)xy + \sum_{i=1}^{\frac{n-1}{2}} k_{3,i} (x^i + x^{-i}) +
\sum_{i=1}^{\frac{n-1}{2}} r_{3,i} x(x^i + x^{-i}) + 
\sum_{i=1}^{\frac{n-1}{2}} (k_{4,i} + r_{4,i}) x(x^i + x^{-i})y.
\end{gather*}
Then $(\Lambda + x \Omega) \in \Delta(\langle x \rangle)$ implies that $k_1 = r_1,\ k_2 = r_2$ and $k_{4, i} = r_{4, i}$, for $i = 1, 2, \dots, \frac{n-1}{2}$. 
Therefore $Der(\f_{2^m} D_{2n}) = \{d_{(\Lambda y + x \Omega y, \Omega)} \}$, where  
\begin{gather*}
\Lambda y + x \Omega y = r_1(1+x)y + \sum_{i=1}^{\frac{n-1}{2}} k_{3,i} (x^i + x^{-i})y + \sum_{i=1}^{\frac{n-1}{2}} r_{3,i} x(x^i + x^{-i})y \\
\text{and} \quad \Omega = r_1 1 + r_2 y + \sum_{i=1}^{\frac{n-1}{2}} r_{3,i} (x^i + x^{-i}) + \sum_{i=1}^{\frac{n-1}{2}} r_{4,i} (x^i + x^{-i})y.
\end{gather*}
Define $B_o = \{d_{x', y'}\}$ where $(x',y') \in \{((1+x)y, 1), \ ((x^i + x^{-i})y, 0), \ (x(x^i + x^{-i})y, x^i + x^{-i}), \ (0, y), \ (0, (x^i + x^{-i})y) \mid i = 1, 2, \dots , \frac{n-1}{2}\}$. 
$B_o$ is a spanning set for $Der(\f_{2^m} D_{2n})$.
The elements of $B_o$ are linearly independent since 
$d_{(\Lambda y + x \Omega y, \Omega)} = d_{(0,0)}$ implies that $r_1 = r_2 = r_{3,i} = r_{4,i} = k_{3,i} = 0$, for $i = 1, 2, \dots, \frac{n-1}{2}$. 

Therefore $Der(\f_{2^m} D_{2n})$ has a basis $B_o = \{d_{x', y'}\}$ where $(x',y') \in \{((1+x)y, 1), \ ((x^i + x^{-i})y, 0), \ (x(x^i + x^{-i})y, x^i + x^{-i}), \ (0, y), \ (0, (x^i + x^{-i})y) \mid i = 1, 2, \dots , \frac{n-1}{2}\}$. Thus $Der(\f_{2^m} D_{2n})$ has dimension $3(\frac{n-1}{2})+2 = \frac{3n+1}{2}$. 
\end{proof}

\begin{lemma}\cite{Rowen}\label{Lemma:RowenDer}
Let $a, c$ be elements of a ring $R$ and let $d_c$ be the map from $R$ to $R$ defined by $d_c(a) = [a, c]=ac - ca$ for all $a\in R$. Then 
\begin{enumerate}
    \item The Lie commutator is anti-symmetric, $[a,b] = -[b,a]$.
    \item the map $d_c$ is an inner derivation for all $c \in R$.
    \item $d_c = 0$ if and only if $c \in Z(R)$.
\end{enumerate}
\end{lemma}

We now give a basis for the set of inner derivations of $\f_{2^m} D_{2n}$.

\begin{theorem}\label{Thm:basisofInn}
The set of inner derivations of $\f_{2^m} D_{2n}$ is an $\f_{2^m}$-vector space with dimension $3\floor{\frac{n-1}{2}}$ and basis 
\[ \left\{d_b \mid b \in \{x^i \mid i=1, 2, \dots, \floor{\tfrac{n-1}{2}}\} \cup \{x^iy \mid i = 0, 1,\dots,2\floor{\tfrac{n-1}{2}}-1\}\right\}.\]
\end{theorem}
\begin{proof}
By Lemma~\ref{Lemma:RowenDer} the Lie commutator is anti-symmetric and so it is symmetric in characteristic 2. 
Let $a, b, c \in \f_{2^m} D_{2n}$. 
Then $d_{a + b}(c) = d_c(a + b) = d_c(a) + d_c(b) = d_a(c) + d_b(c)$ and so the inner derivations of $\f_{2^m} D_{2n}$ are closed under addition. If $k \in \f_{2^m}$, then $k d_b = d_{kb}$ and thus the inner derivations of $\f_{2^m} D_{2n}$ form a vector subspace of $Der(\f_{2^m} D_{2n})$. Let $B = \{x^i \mid i=1, 2, \dots, \floor{\frac{n-1}{2}}\}$  $\cup$ $\{x^iy \mid i = 0, 1,\dots,2\floor{\frac{n-1}{2}}-1\}$.

Case(1) $n$ is even. Write $n = 2c$.
By Lemma~\ref{basisforZKG}, $Z(\f_{2^m} D_{2n})$ is a $(\frac{n}{2}+3)$-dimensional subspace of $\f_{2^m} D_{2n}$ with basis $B_Z = \{1, x^c, x+x^{-1}, x^2+x^{-2}, \dots, x^{(c-1)} + x^{(c+1)}, \sum_{i=0}^{c-1} x^{2i}y, \sum_{i=0}^{c-1} x^{2i+1}y\}$. 
The union of the disjoint sets $B$ and $B_Z$ is a basis for $\f_{2^m} D_{2n}$.

Case(2) $n$ is odd. Write $n = 2c+1$.
By Lemma~\ref{basisforZKG}, $Z(\f_{2^m} D_{2n})$ is a $(\frac{n+3}{2})$-dimensional subspace of $\f_{2^m} D_{2n}$ with basis $B_Z = \{1, x+x^{-1}, x^2+x^{-2}, \dots, x^{c} + x^{-c}, \sum_{i=0}^{2c} x^{i}y\}$. 
Again, the disjoint union of $B$ and $B_Z$ is a basis for $\f_{2^m} D_{2n}$.

Write $a = z_a + \sum_{i=1}^{3\floor{\frac{n-1}{2}}} a_i b_i$, where $z_a \in Z(\f_{2^m} D_{2n})$, $a_i \in \f_{2^m}$ and $b_i \in B$,  for $i = 1, 2, \dots, 3\floor{\frac{n-1}{2}}$. $d_c = 0$ if and only if $c \in Z(\f_{2^m} D_{2n})$ and so 
\[d_a = d_{z_a} + \sum_{i = 1}^{3\floor{\frac{n-1}{2}}}d_{a_ib_i} = \sum_{i = 1}^{3\floor{\frac{n-1}{2}}}d_{a_ib_i}.\]
Therefore the set $\{ d_b \mid b \in B\} $ spans the set of inner derivations of $\f_{2^m} D_{2n}$. 
Moreover, if $\sum_{i = 1}^{3\floor{\frac{n-1}{2}}}d_{a_ib_i} = d_0$ then $\sum_{i = 1}^{3\floor{\frac{n-1}{2}}}a_ib_i \in Z(\f_{2^m} D_{2n})$ which implies that $a_i = 0$, for $i = 1, 2, \dots, 3\floor{\frac{n-1}{2}}$ and so the set $\{ d_b \mid b \in B \} $ forms a basis for the vector space of inner derivations of $\f_{2^m} D_{2n}$.
\end{proof}

The derivation problem asks whether every derivation from $L^1(G)$ to $M(G)$ is inner, where $G$ is a locally compact group and $M(G)$ is the multiplier algebra of $L^1(G)$. 
It was solved by Losert \cite{Losert}. 
We can ask a similar question for finite group algebras. Let $KG$ be a group algebra where both $K$ and $G$ are finite. 
Are all derivations of $KG$ inner?
Theorems~\ref{Thm:BasisDer(fpmD2n)} and \ref{Thm:basisofInn} show that the dimension of $Der(\f_{2^m} D_{2n})$ is greater than the dimension of the inner derivations of $\f_{2^m} D_{2n}$ and so not all derivations of $\f_{2^m} D_{2n}$ are inner. However does there exist an algebra $A \supset KG$ such that all derivations of $KG$ become inner in $A$? Theorem~\ref{thm:derprob} answers this question.

\begin{definition}\cite{Rowen}
Let $R$ be a ring and $\delta$ a derivation of $R$. 
The ring $R[x; \delta] = \left \{ \displaystyle \sum_{i=0}^{n} a_i x^i,\ a_i \in R \right \}$, 
where addition is performed componentwise and multiplication satisfies the relation $xa = ax +~\delta(a)$, for all $a \in R$ is called a \emph{differential polynomial ring}.
\end{definition}

\begin{theorem}\label{thm:derprob}
Let $G$ be a finite group and $KG$ be the group algebra over the finite field $K$. Let $A_d = \displaystyle \frac{KG[x; d]}{(x^2-1)}$, where $d \in Der(KG)$. Then all derivations $d$ of $KG$ are inner on $A_d$.
\end{theorem}
\begin{proof}
Let $D_x$ be the inner derivation of $A_d$ induced by $x$, that is $D_x \colon A_d \to A_d$, defined by $a \mapsto xa - ax$. By the multiplication relation of $A_d$, $xa - ax = d(a)$. Therefore the restriction of $D_x$ to $KG$ is equal to $d$.  
\end{proof}

%% file: Section3_3.tex

\begin{example}
    Let $C_{24} = \langle x \mid x^{24} = 1 \rangle$ and let $d \colon \f_2 C_{24} \to \f_2 C_{24}$ be the derivation defined by $x \mapsto 1+x+x^3+x^4+x^5+x^7+x^9+x^{12}$ (by Theorem~\ref{ThmDerFiniteAbelG} this uniquely defines a derivation). Then by Lemma~\ref{LogTablesTheorem}, $d(x^{2n}) = 0$ and $d(x^{2n+1}) = x^{2n}d(x)$, for $n \in \{0,1, \dots, 11\}$. The image of the group algebra under this derivation is a binary code of length $24$ and dimension $12$. A generator matrix $G_{24}$ of this code is given in Figure~\ref{fig:GenMat2412}.
    
    \begin{figure}[ht]
        \centering
        \caption{Generator matrix of the binary $[24, 12, 8]$ code defined by the derivation $d$.}
        \label{fig:GenMat2412}
    
    {\small
    \[G_{24} = 
    \left[ \begin{array}{cccccccccccccccccccccccc}
    
    1 & 1 & 0 & 1 & 1 & 1 & 0 & 1 & 0 & 1 & 0 & 0 & 1 & 0 & 0 & 0 & 0 & 0 & 0 & 0 & 0 & 0 & 0 & 0\\
    0 & 0 & 1 & 1 & 0 & 1 & 1 & 1 & 0 & 1 & 0 & 1 & 0 & 0 & 1 & 0 & 0 & 0 & 0 & 0 & 0 & 0 & 0 & 0\\
    0 & 0 & 0 & 0 & 1 & 1 & 0 & 1 & 1 & 1 & 0 & 1 & 0 & 1 & 0 & 0 & 1 & 0 & 0 & 0 & 0 & 0 & 0 & 0\\
    0 & 0 & 0 & 0 & 0 & 0 & 1 & 1 & 0 & 1 & 1 & 1 & 0 & 1 & 0 & 1 & 0 & 0 & 1 & 0 & 0 & 0 & 0 & 0\\
    0 & 0 & 0 & 0 & 0 & 0 & 0 & 0 & 1 & 1 & 0 & 1 & 1 & 1 & 0 & 1 & 0 & 1 & 0 & 0 & 1 & 0 & 0 & 0\\
    0 & 0 & 0 & 0 & 0 & 0 & 0 & 0 & 0 & 0 & 1 & 1 & 0 & 1 & 1 & 1 & 0 & 1 & 0 & 1 & 0 & 0 & 1 & 0\\
    1 & 0 & 0 & 0 & 0 & 0 & 0 & 0 & 0 & 0 & 0 & 0 & 1 & 1 & 0 & 1 & 1 & 1 & 0 & 1 & 0 & 1 & 0 & 0\\
    0 & 0 & 1 & 0 & 0 & 0 & 0 & 0 & 0 & 0 & 0 & 0 & 0 & 0 & 1 & 1 & 0 & 1 & 1 & 1 & 0 & 1 & 0 & 1\\
    0 & 1 & 0 & 0 & 1 & 0 & 0 & 0 & 0 & 0 & 0 & 0 & 0 & 0 & 0 & 0 & 1 & 1 & 0 & 1 & 1 & 1 & 0 & 1\\
    0 & 1 & 0 & 1 & 0 & 0 & 1 & 0 & 0 & 0 & 0 & 0 & 0 & 0 & 0 & 0 & 0 & 0 & 1 & 1 & 0 & 1 & 1 & 1\\
    1 & 1 & 0 & 1 & 0 & 1 & 0 & 0 & 1 & 0 & 0 & 0 & 0 & 0 & 0 & 0 & 0 & 0 & 0 & 0 & 1 & 1 & 0 & 1\\
    0 & 1 & 1 & 1 & 0 & 1 & 0 & 1 & 0 & 0 & 1 & 0 & 0 & 0 & 0 & 0 & 0 & 0 & 0 & 0 & 0 & 0 & 1 & 1\\
    \end{array} \right] 
    \]}
    
    \end{figure}
        Permuting the columns of $G_{24}$ using the permutation 
    \[(6,19,12,10,11,22,8,21,15,16,18,9,24,13,20)(7,23,17,14)\] 
    and then transforming it to reduced row echelon form produces the matrix given as the generator of the extended binary Golay code in \cite{Hill}. 
    So the image of $\f_2C_{24}$ under the derivation is equivalent to the extended binary Golay $[24,12,8]$ code.
    It has minimum distance $8$ and is a doubly even and self dual extremal code. 
\end{example}
    
\begin{figure}[ht]
        \centering
        \caption{The right hand block of a generator matrix of the binary $[48, 24, 12]$ code defined by the derivation $\delta$.}
        \label{fig:GenMat4824}
    
    {\small
    \[A = 
    \left[ \begin{array}{cccccccccccccccccccccccc}
    1 & 0 & 0 & 1 & 0 & 0 & 0 & 1 & 1 & 1 & 0 & 0 & 0 & 1 & 0 & 0 & 1 & 1 & 0 & 1 & 1 & 0 & 0 & 1 \\
    0 & 1 & 1 & 0 & 0 & 1 & 1 & 0 & 1 & 1 & 0 & 0 & 1 & 0 & 0 & 0 & 1 & 1 & 1 & 0 & 0 & 0 & 1 & 0 \\
    0 & 1 & 0 & 0 & 0 & 0 & 1 & 0 & 1 & 0 & 1 & 1 & 1 & 0 & 0 & 1 & 1 & 1 & 0 & 1 & 0 & 1 & 0 & 0 \\
    1 & 0 & 0 & 0 & 1 & 0 & 0 & 0 & 0 & 1 & 1 & 1 & 0 & 1 & 1 & 0 & 1 & 1 & 1 & 0 & 0 & 0 & 0 & 1 \\
    0 & 0 & 0 & 1 & 0 & 0 & 0 & 0 & 1 & 0 & 1 & 0 & 1 & 1 & 1 & 0 & 0 & 1 & 1 & 1 & 0 & 1 & 0 & 1 \\
    0 & 1 & 0 & 0 & 0 & 1 & 0 & 0 & 1 & 1 & 0 & 1 & 0 & 1 & 0 & 1 & 0 & 1 & 0 & 1 & 1 & 0 & 1 & 0 \\
    0 & 1 & 1 & 0 & 0 & 0 & 1 & 0 & 1 & 1 & 1 & 0 & 0 & 0 & 1 & 1 & 0 & 1 & 1 & 1 & 1 & 1 & 1 & 1 \\
    1 & 0 & 0 & 0 & 0 & 0 & 0 & 0 & 1 & 1 & 1 & 1 & 0 & 0 & 0 & 1 & 1 & 0 & 0 & 0 & 1 & 1 & 1 & 1 \\
    1 & 1 & 1 & 0 & 1 & 1 & 1 & 1 & 1 & 0 & 1 & 1 & 0 & 1 & 0 & 0 & 1 & 1 & 1 & 0 & 0 & 1 & 0 & 0 \\
    1 & 1 & 0 & 1 & 0 & 1 & 1 & 1 & 0 & 0 & 1 & 1 & 0 & 0 & 0 & 0 & 0 & 1 & 0 & 1 & 1 & 0 & 0 & 0 \\
    0 & 0 & 1 & 1 & 1 & 0 & 1 & 1 & 1 & 1 & 1 & 0 & 1 & 1 & 0 & 1 & 0 & 0 & 1 & 1 & 1 & 0 & 0 & 1 \\
    0 & 0 & 1 & 1 & 0 & 1 & 0 & 1 & 1 & 1 & 0 & 0 & 1 & 1 & 0 & 0 & 0 & 0 & 0 & 1 & 0 & 1 & 1 & 0 \\
    0 & 1 & 1 & 0 & 1 & 0 & 0 & 0 & 0 & 0 & 1 & 1 & 0 & 0 & 1 & 1 & 1 & 0 & 1 & 0 & 1 & 1 & 0 & 0 \\
    1 & 0 & 0 & 1 & 1 & 1 & 0 & 0 & 1 & 0 & 1 & 1 & 0 & 1 & 1 & 1 & 1 & 1 & 0 & 1 & 1 & 1 & 0 & 0 \\
    0 & 0 & 0 & 1 & 1 & 0 & 1 & 0 & 0 & 0 & 0 & 0 & 1 & 1 & 0 & 0 & 1 & 1 & 1 & 0 & 1 & 0 & 1 & 1 \\
    0 & 0 & 1 & 0 & 0 & 1 & 1 & 1 & 0 & 0 & 1 & 0 & 1 & 1 & 0 & 1 & 1 & 1 & 1 & 1 & 0 & 1 & 1 & 1 \\
    1 & 1 & 1 & 1 & 0 & 0 & 0 & 1 & 1 & 0 & 0 & 0 & 1 & 1 & 1 & 1 & 0 & 0 & 0 & 0 & 0 & 0 & 0 & 1 \\
    1 & 1 & 1 & 1 & 1 & 1 & 1 & 0 & 1 & 1 & 0 & 0 & 0 & 1 & 1 & 1 & 0 & 1 & 0 & 0 & 0 & 1 & 1 & 0 \\
    0 & 1 & 0 & 1 & 1 & 0 & 1 & 0 & 1 & 0 & 1 & 0 & 1 & 0 & 1 & 1 & 0 & 0 & 1 & 0 & 0 & 0 & 1 & 0 \\
    1 & 0 & 1 & 0 & 1 & 1 & 1 & 0 & 0 & 1 & 1 & 1 & 0 & 1 & 0 & 1 & 0 & 0 & 0 & 0 & 1 & 0 & 0 & 0 \\
    1 & 0 & 0 & 0 & 0 & 1 & 1 & 1 & 0 & 1 & 1 & 0 & 1 & 1 & 1 & 0 & 0 & 0 & 0 & 1 & 0 & 0 & 0 & 1 \\
    0 & 0 & 1 & 0 & 1 & 0 & 1 & 1 & 1 & 0 & 0 & 1 & 1 & 1 & 0 & 1 & 0 & 1 & 0 & 0 & 0 & 0 & 1 & 0 \\
    0 & 1 & 0 & 0 & 0 & 1 & 1 & 1 & 0 & 0 & 0 & 1 & 0 & 0 & 1 & 1 & 0 & 1 & 1 & 0 & 0 & 1 & 1 & 0 \\
    1 & 0 & 0 & 1 & 1 & 0 & 1 & 1 & 0 & 0 & 1 & 0 & 0 & 0 & 1 & 1 & 1 & 0 & 0 & 0 & 1 & 0 & 0 & 1 \\
    \end{array} \right] 
    \]}
    
    \end{figure}

\begin{example}
    Let $C_{48} = \langle x \mid x^{48} = 1 \rangle$ and $\delta \colon \f_2 C_{48} \to \f_2 C_{48}$ be the derivation defined by 
    \[x \mapsto 1+x^{24}+x^{27}+x^{31}+x^{32}+x^{33}+x^{37}+x^{40}+x^{41}+x^{43}+x^{44}+x^{47}.\] 
    Again by Theorem~\ref{ThmDerFiniteAbelG} this uniquely defines a derivation of $\f_2 C_{48}$.
    The image of the group algebra under this derivation is a binary $[48, 24, 12]$ doubly even self dual code (verified using GAP 4.8.6 \cite{GAP4}). It is equivalent to the extended binary quadratic residue code of length $48$ \cite{Houghten2003}. A generator matrix for this code is given by the block matrix $\left[ I_{24} \mid A \right]$, where $I_{24}$ is the identity of the ring of $24 \times 24$ matrices over $\f_2$ and $A$ is the matrix given in Figure~\ref{fig:GenMat4824}.
\end{example}